\documentclass{amsart}

\usepackage[leqno]{amsmath}
\usepackage{amssymb,amsfonts,xfrac,MnSymbol}
\usepackage{doc}
\usepackage{amsthm}
\usepackage{cite}
\usepackage{hyperref}
\usepackage{todonotes}
\usepackage{enumitem}
\usepackage{graphicx}
\usepackage{tikz-cd, tikz}
\usepackage{color}
\usepackage{import}
\usepackage{csquotes}

\numberwithin{equation}{section}

\newtheorem{theorem}{Theorem}[section]

\newtheorem{proposition}[theorem]{Proposition}
\newtheorem{lemma}[theorem]{Lemma}
\newtheorem{corollary}[theorem]{Corollary}

\newtheorem*{theorem*}{Theorem}
\newtheorem*{claim*}{Claim}
\newtheorem*{proposition*}{Proposition}
\newtheorem*{lemma*}{Lemma}
\newtheorem*{corollary*}{Corollary}

\newtheorem{theoremA}{Theorem}

\newtheorem{corA}[theoremA]{Corollary}

\theoremstyle{definition}
\newtheorem{definition}[theorem]{Definition}
\newtheorem{observation}[theorem]{Observation}
\newtheorem{remark}[theorem]{Remark}

\newtheorem{question}[theorem]{Question}

\newtheorem*{definition*}{Definition}
\newtheorem*{observation*}{Observation}
\newtheorem*{remark*}{Remark}
\newtheorem*{example*}{Example}
\newtheorem*{question*}{Question}
\newtheorem*{exercise*}{Exercise}
\newtheorem*{fact*}{Fact}
\newtheorem*{notation*}{Notation}

\newcommand{\bbN}{\mathbb{N}}

\newcommand{\bbQ}{\mathbb{Q}}
\newcommand{\bbR}{\mathbb{R}}

\newcommand{\bbZ}{\mathbb{Z}}

\newcommand{\calG}{\mathcal{G}}

\newcommand{\vE}{\vec{E}}
\newcommand{\DG}{\mathbf{d}_\calG}
\newcommand{\DD}{\mathbf{d}}

\newcommand{\restrict}{\upharpoonright}

\newcommand{\epi}{\twoheadrightarrow}

\newcommand{\normalin}{\lhd}

\newcommand{\ii}{^{-1}}

\newcommand{\gen}[1]{\left< #1 \right>}

\newcommand{\piG}{\pi_1(\calG,T)}
\newcommand{\barpiG}{\overline{\pi}_1(\calG,T)}

\DeclareMathOperator{\id}{id}

\DeclareMathOperator{\Span}{Span}

\DeclareMathOperator{\Trans}{Trans}
\DeclareMathOperator{\Acts}{Acts}

\DeclareMathOperator{\Sym}{Sym}

\title{Virtually free groups are stable in permutations}
\author{Nir Lazarovich}
\thanks{NL is supported by the Israel Science Foundation (grant No. 1562/19) and the German-Israeli Foundation for Scientific Research and Development.}
\author{Arie Levit}
\date{}

\begin{document}

\maketitle

\begin{abstract}
We prove that finitely generated virtually free groups are stable in permutations. As an application,  we show that  almost-periodic almost-automorphisms of labelled graphs are close to periodic automorphisms.
\end{abstract}

\section{Introduction}


A finitely generated group $G$ is called stable in permutations (in short P-stable) if every almost action of $G$ on a finite set is close to an honest action (see \S\ref{sec: relative stability} for definitions).  For the ubiquitous class of sofic groups, the property of P-stability can be seen as a stronger form of residual finiteness \cite{glebsky2009almost}. 
    Our main result is:


\begin{theoremA}\label{thm: main thm}
Every finitely generated virtually free group is P-stable.
\end{theoremA}

It is trivially true that free groups are P-stable. 
But while residual finiteness is   preserved under passing to finite index subgroups (or rather to any subgroup), this fact is not clear in general for P-stability. 

To the best of our knowledge Theorem \ref{thm: main thm} gives the first examples of P-stable groups which are not free products of P-stable amenable groups. Note that while fundamental groups of closed orientable surfaces are known to be flexibly P-stable  \cite{lazarovich2019surface}, it is not clear if these groups are  P-stable in the strict sense.


As a special case of Theorem \ref{thm: main thm} we answer  the following question of  Lubotzky.

\begin{corollary}
The modular group $\mathrm{SL}_2(\bbZ)$ is P-stable.
\end{corollary}

Interestingly P-stability is not, generally speaking, preserved under direct products --- for example,  the groups $F_2 \times \bbZ$ are not P-stable \cite{ioana2020stability}. This phenomenon is to be contrasted with the fact that the product groups $F_2\times (\bbZ/n\bbZ)$ are P-stable for all $n \in \bbN$, as follows from   Theorem \ref{thm: main thm}. As a consequence of the P-stability of these groups we are able to deduce:

\begin{corA}
\label{cor:graph automorphisms are locally testable}
Fix some $d,n \in \bbN$. Let $F_d$ be the free group of rank $d$ and $\mathcal{G}_d$ be the family of finite labelled Schreier graphs of  $F_d$. Then for every graph $\Gamma \in \mathcal{G}_d$ and every $\delta$-almost automorphism $\alpha$ of $\Gamma$ of $\delta$-almost order $n$, there is a graph $\Gamma' \in \mathcal{G}_d$ on the same vertex set as $\Gamma$ and $O(\delta)$-close to $\Gamma$  and an  automorphism $\alpha'$ of $\Gamma'$ which is $O(\delta)$-close to $\alpha$ and has order $n$.
\end{corA}

More details and a precise statement of Corollary \ref{cor:graph automorphisms are locally testable} can be found in  \S\ref{sec:graph automorphisms} below. 

\subsection*{Stable epimorphisms}

Stallings theorem on ends of groups \cite{stallings1968torsion, stallings1972group} implies that a finitely generated group $G$ is virtually free   if and only if $G$ is isomorphic to the fundamental group $\piG$ of a finite graph of groups $\calG$ with finite vertex groups with respect to some maximal spanning tree $T$ (see \S\ref{sec:fundamental group} for the definition of $\piG$).

Naturally associated to the graph of groups $\calG$ and the maximal spanning tree $T$ there is another group $\barpiG$ admitting a quotient map $\barpiG \to \piG$. This group  is isomorphic to the free product of the vertex groups of $\calG$ with the topological fundamental group of the underlying graph of $\calG$.  As finite groups are  P-stable  \cite{glebsky2009almost} it follows immediately that the group $\barpiG$ is P-stable. 

Motivated by this, we introduce a relative notion of P-stable epimorphisms, see Definition \ref{def:stable epi}. In particular, a finitely generated group $G$ is P-stable in the usual sense if and only if the natural epimorphism from the free group in the generators of $G$ onto the group $G$  is P-stable.  Theorem \ref{thm: main thm} is thereby reduced to  the following statement, to which the major part of this work is dedicated.

\begin{theorem}
\label{thm:quotient map is stable}
The epimorphism $\barpiG \to \piG$ is P-stable.
\end{theorem}

A detailed outline of the proof of Theorem \ref{thm:quotient map is stable} can be found in \S\ref{subsec: outline of proof} below, after the necessary definitions and notations are set in place. 

\section{P-stable epimorphisms}
\label{sec: relative stability}

Let $X$ be a finite set. Consider the normalized Hamming distance $d_X$ on the symmetric group $\Sym(X)$ given by
$$ d_X(\sigma_1,\sigma_2) = \frac{1}{|X|} | \{ x \in X \: : \: \sigma_1(x) \neq \sigma_2(x) \} |$$
for all pairs  $\sigma_1, \sigma_2 \in \Sym(X)$. Note that the metric $d_X$ is bi-invariant.

Let $\overline{G}$ be a group with finite generating set $S$. Define a metric $d_{X,S}$  on the   set $\mathrm{Hom}(\overline{G}, \Sym(X))$ of all group homomorphisms $\rho : \overline{G} \to \Sym(X)$ by
$$d_{X,S}(\rho,\rho')=\sum_{s\in S} d_X(\rho(s),\rho(s'))$$
for each    pair   $\rho,\rho' \in\mathrm{Hom}(\overline{G}, \Sym(X))$.

Let $ N \normalin \overline{G}$ be a normal subgroup normally generated   by some  finite subset   $R \subset \overline{G}$. Denote $G = \overline{G} / N$. We say 
that an action  $\rho : \overline{G} \to \Sym(X)$ is   a  \emph{$\delta$-almost $G$-action}  
 if
$$ \sum_{r \in R} d_X( \rho(r), \id ) < \delta.$$
This terminology is justified by the observation that $\rho$ is an honest $G$-action if and only if it is a $\delta$-almost $G$-action with respect to $\delta = 0$. Note that strictly speaking this notion depends on fixing the normal generating set $R$. 

\begin{definition}
\label{def:stable epi}
The  epimorphism $\phi:\overline{G} \to G $ is \emph{P-stable}  if for every $\varepsilon > 0$ there is a $\delta > 0$ such that for every $\delta$-almost $G$-action $\rho : \overline{G} \to \Sym(X)$ there is a $G$-action $\rho' : G \to \Sym(X)$ with $d_{X,S}(\rho,\rho'\circ \phi) < \varepsilon$.
\end{definition}

\begin{lemma}
The  P-stability of the epimorphism $\phi : \overline{G} \to G$ is a well-defined notion (i.e. it is independent of the choices of the finite sets $S$ and $R$).
\end{lemma}
\begin{proof}
It is easy to see that if $S_1$ and $S_2$ are  two finite generating sets for the group $G$ then the resulting metrics $d_{X,S_1}$ and $d_{X,S_2}$ on the set $\mathrm{Hom}(\overline{G}, \Sym(X))$ are bi-Lipschitz equivalent. A similar argument, taking into account the bi-invariance of the normalized Hamming metric $d_X$, shows that if $R_1$ and $R_2$ are two finite normal generating sets for the subgroup $N \normalin \overline{G}$ then there is a constant $C = C(R_1, C_2) > 1$ such that
$$ C^{-1} \sum_{r \in R_2} d_X( \rho(r), \id ) \le \sum_{r \in R_1} d_X( \rho(r), \id ) \le C \sum_{r \in R_2} d_X( \rho(r), \id ).$$
The conclusion follows from these observations.
\end{proof}

Let $H$ be any group admitting a finite generating set $S$ and $F(S)$ be the free group in the generators $S$. Observe that the natural homomorphism $F(S) \to H$ is P-stable if and only if the group $H$ is $P$-stable in the usual sense.
\begin{remark}
Every split epimorphism is P-stable.
\end{remark} 


The following follows immediately from Definition \ref{def:stable epi}.

\begin{lemma}\label{lem: transitivity  of relative stability}
Let $\overline{\overline{G}} \overset{\phi} \epi \overline{G} \overset{\psi} \epi G$ be a sequence of epimorphisms with normally finitely generated kernels. If $\phi$ and $\psi$ are P-stable then $\psi \circ \phi$ is P-stable.
\end{lemma}

We   have occasion to use Lemma \ref{lem: transitivity  of relative stability} only in the following special form:   if the group $\overline{G}$ is $P$-stable   and $\phi : \overline{G} \epi G$ is a $P$-stable epimorphism then the group $G$ is $P$-stable.

\begin{remark}
It seems an interesting problem to look for other non-trivial instances of P-stable epimorphisms.
\end{remark}

\section{The fundamental group of a graphs of groups}
\label{sec:fundamental group}

We recall the definition of the fundamental group of a graph of groups and in particular list its defining relations. This is followed by a detailed sketch of proof for our Theorem \ref{thm: main thm} a well as for the \enquote{relative} Theorem \ref{thm:quotient map is stable}. Lastly we introduce some useful asymptotic notations.

\subsection*{Graphs of groups}
We use Serre's notation for graphs \cite{serre1977arbres}. 
In this notation, a graph $\Gamma$ consists of a set of vertices $V(\Gamma)$ and a set of edges $E(\Gamma)$. Each edge $e \in E(\Gamma)$ has an origin   $o(e) \in V(\Gamma)$ and a terminus   $t(e) \in V(\Gamma)$. Moreover each edge $e \in E(\Gamma)$ admits a distinct opposite edge   $\bar{e} \in E(\Gamma)$ that satisfies $\bar{\bar{e}} = e$,  $o(\bar{e}) = t(e)$ and $t(\bar{e}) = o(e)$.   Every pair of \enquote{oriented}
edges $\{e,\bar{e}\} \subset E(\Gamma)$ represents a single 
 \enquote{geometric} edge. 
  An orientation of the graph $\Gamma$ is a subset $\vec{E}(\Gamma) \subset E(\Gamma)$ containing exactly a single edge from each pair $\{e,\bar{e}\}$.

\begin{definition}
A \emph{graph of groups} $\calG$ is 
$$\calG = (\Gamma, \{G_v\}_{v\in V(\Gamma)}, \{G_e\}_{e\in E(\Gamma)}, \{i_e:G_e\to G_{t(e)}\}_{e\in E(\Gamma)})$$ 
where $\Gamma$ is a connected graph, $G_v$ is a vertex group for all $v \in V(\Gamma)$, $G_e$ is an edge group for all edges $e \in E(\Gamma)$ with $G_e = G_{\bar{e}}$ and $\iota_e : G_e \to G_{t(e)}$ are injective homomorphisms.
\end{definition}

Let $\calG$ be a graph of groups. Fix an orientation $\vec{E}(\Gamma) $ and a maximal spanning tree $T \subset \Gamma$. 
Consider the group $\barpiG$ defined as the  free product
$$ \overline{\pi}_1(\calG, T) = *_{v\in V(\Gamma)} G_v * F(\{s_e\}_{e\in \vec{E}(\Gamma) })$$
where $F(\cdot)$ denotes the free group over the given basis. It will be convenient to consider  the following generating set 
\[S_\calG=\bigcup_{v\in V(\Gamma)} G_v \cup \{s_e\}_{e\in \vec{E}(\Gamma)}.\]

\begin{definition}
The \emph{fundamental group} $\pi_1(G,T)$ of the graph of groups $\calG$ with respect to the subtree $T$ is the quotient of the free product $\barpiG$ by the normal subgroup generated   by the relations
$$
R_\calG = \begin{cases}
    s_e = 1  & \forall e \in \vec{E}(T), \\
    s_e\ii i_e(g_e) s_e = i_{\bar{e}}(g_e)   &\forall e\in \vec{E}(\Gamma), g_e\in G_e.
\end{cases}
$$
\end{definition}

\begin{remark}
The fundamental group $\piG$ as well as the group $\barpiG$ are independent of the choice of maximal spanning tree $T$ up to isomorphism \cite[I.\S5]{serre1977arbres}.
\end{remark}

For the remainder of the paper we will assume that $\calG$ is a finite graph of groups, with finite vertex groups. In particular
$S_\calG$ is a finite generating set for the group $\barpiG$.

\subsection*{Outline of the proof and of the paper}\label{subsec: outline of proof}

Note that the group $\overline{\pi}_1(\calG,T)$   is a free product of finite groups and of a free group. As such $\overline{\pi}_1(\calG,T)$ is easily seen to be P-stable  \cite{glebsky2009almost}. 
In light of Lemma \ref{lem: transitivity  of relative stability} and the remarks following it,  our main result Theorem \ref{thm: main thm} follows immediately from Theorem \ref{thm:quotient map is stable} of the introduction. In other words it suffices to show that  the epimorphism $\overline{\pi}_1(\calG, T) \to \pi_1(\calG, T)$ is P-stable.

Towards this goal consider some $\delta$-almost $\piG$-action $\rho :\barpiG\to \Sym(X)$. In particular $\rho$ restricts to  actions of the finite vertex groups $G_v$. For   $\rho$ to factorize through the fundamental group $\piG$ it is necessary that for every edge $e \in \vec{E}(\Gamma)$ the two actions $\rho\circ i_e$ and $\rho\circ i_{\bar{e}}$ of the edge group $G_e$ are isomorphic.

It is clear that the isomorphism type of an action of a finite group on a finite set is characterized by the number of occurrences of each of its finitely many transitive action types. In \S\ref{sec: from actions to linear algebra}, we show how to represent this  data using a vector in some canonical $\bbZ$-module associated to the  group.
The restriction maps $\rho|_{G_v} \mapsto (\rho\circ i_e)|_{G_e}$ define a $\bbZ$-linear map $\DG$ between the respective $\bbZ$-modules. The above mentioned condition (that the two actions $\rho\circ i_e$ and $\rho\circ i_{\bar{e}}$ of the edge group $G_e$ are isomorphic) can be described as the kernel of this  $\bbZ$-linear map $\DG$. Lastly, the fact that $\rho$ is a $\delta$-almost action of $\piG$ translates to having  a small image under the   map $\DG$. 

 In \S\ref{sec: linear cones}, we show that \enquote{$\bbZ$-linear maps are stable} in the following sense: an exact $\bbZ$-solution to a linear system of equations and inequalities can be found nearby a $\delta$-almost solution. This is applied  to the  linear map $\DG$. An exact $\bbZ$-solution represents an isomorphism type of an action of $\barpiG$ on a finite set of the same size as $X$, which is statistically close to $\rho$ and admits refinements to   well-defined isomorphism types of actions of the edge groups.

Finally, in \S\ref{sec: from linear algebra to actions} we show how given a $\delta$-almost action $\rho$, and a nearby exact $\bbZ$-solution to the corresponding linear system of equations, one can find a nearby action $\rho'$  factoring via $\piG$.

We will make repeated use of the finiteness of vertex and edge groups via:
\begin{observation}\label{obs: existence of invariant subsets}
Let $G$ be a finite group. If  $G$ acts on a finite set $X$ and  $Y\subset X$ then there exists a $G$-invariant subset $Y'\subseteq Y$ such that $|X-Y'| \le |G| |X-Y|$. 
\end{observation}
\subsection*{Notations}

We will need to consider inequalities involving   quantities depending on the graph of groups $\calG$ in question (such as the number of vertices or edges,  the sizes of the vertex groups $G_v$, etc.). To avoid cumbersome formulas it would be convenient to introduce the following asymptotic notation. 

We write $A\prec_\bullet B$ if there exists a constant $c=c(\bullet)$ such that $A\le cB$. We omit the subscript when it is clear from the context.  

\section{Set of actions on finite sets}
\label{sec: from actions to linear algebra}
Let $G$ be any group. Let $\Acts(G)$ denote the set of all actions of the group $G$ on finite sets considered up to isomorphism. Similarly let $\Trans(G)$ be the set of transtive actions of $G$ on finite sets considered up to isomorphism. 

Every action $ \rho: G\to \Sym(X)$ on some finite set $X$ can be decomposed into a disjoint union of its finitely many orbits $O_1,\ldots,O_n \subseteq X$. The restriction $\rho\restrict _{O_i} :G \to \Sym(O_i)$  of $\rho$ to each orbit $O_i$ is transitive for all $i=1,\ldots,n$. 
The isomorphism class of the action $\rho$ is determined by counting the isomorphism classes of its restricted actions $\rho\restrict_{O_i}$ with multiplicity.

This observation  enables us to identify the set of actions $\Acts(G)$ with a non-negative cone in the free $\bbZ$-module $\Lambda_G$ with basis $\Trans(G)$, namely
$$\Lambda_G= \bigoplus_{\rho\in \Trans(G)} \bbZ\rho.$$
 More precisely, given an action $\rho \in \Acts(G)$ we define
 $$  \rho^\sharp =\sum_{O\in G\backslash X} \rho \restrict _O \in \Lambda_G.$$
The correspondence $\rho \mapsto \rho^\sharp$ is injective and its image $\Acts(G)^\sharp$ in $\Lambda_G$ is  the non-negative cone
$$\Lambda_G^+ :=  \{ (\lambda_\rho)_{\rho \in \Trans(G)} \: : \: \lambda_\rho \ge 0 \}.$$

We observe that the correspondence $\rho\to \rho^\sharp$ is additive in the following sense:   any two actions $\rho_1,\rho_2 \in \Acts(G)$ with $\rho_i:G\to\Sym(X_i)$ for  $i\in \{1,2\}$ satisfy $$(\rho_1
\coprod \rho_2)^\sharp = \rho_1^\sharp + \rho_2^\sharp$$ where $
\rho_1 \coprod \rho_2:G\to \Sym(X_1\coprod X_2)$ is the disjoint union of   $\rho_1$ and $\rho_2$.

We find it convenient to introduce a norm $\|\cdot\|_G$ on the $\bbZ$-module $\Lambda_G$ by
$$ \|\lambda\|_G = \sum_{\substack{\rho \in \Trans(G)\\\rho : G \to \Sym(X_\rho) }} |\lambda_\rho| \cdot |X_\rho|   \quad \quad  \forall \lambda = (\lambda_\rho) \in \Lambda_G.$$
This norm is chosen in such a way that every action $\rho \in \Acts(G)$ with  $\rho:G\to\Sym(X)$ satisfies $\|\rho^\sharp\|_G = |X|$.


\subsubsection*{Pullback on set of actions.} 
Let $H$ be any group admitting a homomorphism    $i :H\to G$. There is a pullback map $i^*$ on the corresponding sets of isomorphism classes of actions on finite sets is given by
$$ i^* : \Acts(G) \to \Acts(H), \quad i^* \rho = \rho \circ i \quad \forall \rho \in \Acts(G).$$
Allowing for a slight abuse of notation, we also let  $i^*$   denote   the resulting $\bbZ$-linear map $i^* : \Lambda_G \to \Lambda_H$ defined in terms of the basis   by
$$i^*(\rho ^\sharp ) = (\rho\circ i)^\sharp \quad \rho\in\Trans(G). $$

\begin{observation}\label{obs: extending actions from subgroups} 
Let $\phi:H\to\Sym(X)$ be a group action such that $\phi^\sharp = i^*(\lambda)$ for some $\lambda\in \Lambda_G$. 
 Then there exists a group  action $\rho:G\to\Sym(X)$ satisfying $\rho^\sharp = \lambda$ and $\rho\circ i = \phi$.
\end{observation}


 
\subsection*{Set of actions for a graph of groups.}
We extend notions introduced above to the setting of  graphs of groups.
Recall that
$$\calG = (\Gamma, \{G_v\}_{v\in V(\Gamma)}, \{G_e\}_{e\in E(\Gamma)}, \{i_e:G_e\to G_{t(e)}\}_{e\in E(\Gamma)})$$
is a finite graph of groups with finite vertex groups. We define the $\bbZ$-modules 
$$ \Lambda_V = \bigoplus_{v\in V(\Gamma)} \Lambda_{G_v} \quad \text{and} \quad \Lambda_E = \bigoplus_{e\in \vec{E}(\Gamma)} \Lambda_{G_e}$$
and   the respective positive cones 
$$\Lambda_V^+ = \bigoplus_{v\in V(\Gamma)} \Lambda_{G_v}^+ \quad \text{and} \quad \Lambda_E^+ = \bigoplus_{e\in \vec{E}(\Gamma)} \Lambda_{G_e}^+.$$
It will be convenient to consider the $\bbZ$-modules with the   norms
$$ \|\cdot\|_{V} = \frac{1}{|V(\Gamma)|} \sum_{v\in V(\Gamma)} \|\cdot\|_{G_v}, \quad \|\cdot\|_{E} = \frac{1}{|\vE(\Gamma)|} \sum_{e \in \vE(\Gamma)} \|\cdot\|_{G_e}$$
where $\|\cdot\|_{G_v}$ and $\|\cdot\|_{G_e}$ are the norms defined on the $\bbZ$-modules $\Lambda_{G_v}$ and $\Lambda_{G_e}$ as above.

Let $\DG : \Lambda_V \to \Lambda_E$ be the $\bbZ$-linear map defined on each direct summand $\Lambda_{G_v}$ of the $\bbZ$-module $\Lambda_V$ by
$$ (\DG)|_{{\Lambda_{G_v}}} = \sum_{e:t(e) = v} i_e^* - \sum_{e:o(e) = v}  i^*_{\bar{e}}.$$
In other words,   the image of the vector $\lambda=(\lambda_v)_v\in \Lambda_V$ under the $\bbZ$-linear map $\DG$ in each coordinate $e\in \vE(\Gamma)$ is given  by 
$$ (\DG(\lambda))_e = i^*_e (\lambda_{t(e)})- i^*_{\bar{e}} (\lambda_{o(e)}).$$

\subsection*{Actions of the group $\barpiG$  on finite sets}
Let $X$ be a fixed finite set. Given an action $\rho : \barpiG \to \Sym(X)$   we denote (abusing our previous notations)
$$\rho^\sharp \in \Lambda_V, \quad (\rho^\sharp)_v = (\rho|_{G_v})^\sharp \quad \forall v \in V.$$
Note that $\|\rho^\sharp\|_V =  |X|$. Moreover the vector $\rho^\sharp $ depends only on the restrictions of the action $\rho$ to   the vertex groups $G_v$'s but not to the free factor $F(\{s_e\})$.

\begin{proposition}
\label{prop:honest action is in kernel}
If the action $\rho : \barpiG \to \Sym(X)$   factors through $\piG $ then $\rho^\sharp \in \ker \DG$.
\end{proposition}
 
\begin{proof}
The $\Lambda_{G_e}$-coordinate of  the image of the vector $\rho^\sharp$ under the $\bbZ$-linear map $\DG$ for any fixed oriented edge $e\in \vE(\Gamma)$ is given by
$$ (\DG(\rho^\sharp))_e   = i^*_e ((\rho|_{G_{t(e)}})^\sharp)- i^*_{\bar{e}} ((\rho|_{G_{o(e)}})^\sharp) = (\rho\circ i_e)^\sharp - (\rho\circ i_{\bar{e}})^\sharp   \in \Lambda_{G_e}.$$
The two actions $\rho\circ i_e$ and $\rho\circ i_{\bar{e}}$ of the edge group $G_e$ on the finite set $X$ are conjugate via the permutation  $\rho(s_e)$. Therefore $  (\rho\circ i_{e})^\sharp = (\rho\circ i_{\bar{e}})^\sharp$ so  that the $\Lambda_{G_e}$-coordinate in question vanishes. This concludes the proof as the oriented edge $e \in \vE(\Gamma)$ was arbitrary.
\end{proof}

We remark that the converse of Proposition \ref{prop:honest action is in kernel} is also true, in the sense that if 
a vector $\lambda \in  \Lambda_V^+$ is in $\ker \DG$  then there exists  a finite set $Y$ with $\|\lambda\|_V = |Y|$ and  some action $\rho:\pi_1(\calG,T)\to \Sym(Y)$ such that $\rho^\sharp=\lambda$.
We will need a much sharper version of this fact proved in Proposition \ref{prop: ker implies action} below.

\begin{proposition}
\label{prop:a delta almost action gives almost kernel}
Let $\rho : \barpiG \to \Sym(X)$ be an action. If $\rho$ is a $\delta$-almost $\piG$-action then $$\|\DG (\rho^\sharp)\|_E \prec_\calG \delta \|\rho^\sharp\|_V .$$ 
\end{proposition}

\begin{proof}
Fix an oriented edge $e \in \vec{E}(\Gamma)$ with $t(e) = u$ and $  o(e)=v$. 
For each group element $g \in G_e$ consider the subset
$$ X_g = \{ x \in X \: : \: \rho(s_e\ii i_e(g) s_e) (x) = \rho(i_{\bar{e}}(g)) (x) \}.$$
The assumption that $\rho$ is a $\delta$-almost  $\piG$-action implies that  $|X_g| \ge (1-\delta)|X| $. 
Denote  $X_e = \bigcap_{g \in G_e} X_g  $ so that  $|X_e| \ge (1-\delta|G_e|)|X|$ and
$$ \rho(s_e\ii i_e(g) s_e)(x) = \rho(i_{\bar{e}}(g)) (x) \quad \forall x \in X_e, g \in G_e.$$
According to Observation \ref{obs: existence of invariant subsets} there is some  $i_{\bar{e}}(G_e)$-invariant subset $Y_e \subset X_e$ satisfying $|Y_e| \ge (1- \delta |G_e|^2) |X|$. Note that the set $\rho(s_e)(Y_e)$ is $i_e(G_e)$-invariant. Moreover, the two actions $(\rho \circ i_e)\restrict _{\rho(s_e)(Y_e)}$ and $(\rho \circ i_{\bar{e}})\restrict _{Y_e}$ of the group $G_e$ are isomorphic (via conjugation by the permutation $\rho(s_e)$).

To simplify our notations let $\rho_u = \rho|_{G_u}$ and $\rho_v = \rho|_{G_v}$ for the remainder of this proof. The previous paragraph implies that
$$((\rho_u \circ i_e)  \restrict_{\rho(s_e)(Y_e)})^\sharp = ((\rho_v \circ i_{\bar{e}})  \restrict_{ Y_e})^\sharp.$$
The norm of the   coordinate of the vector $\DG(\rho^\sharp)$ corresponding to the edge $e$ is given by
\begin{align*}
\begin{split}
    \|(\DG (\rho^\sharp)_e \|_{G_e} &= \|i_e^* (\rho_u^\sharp) - i_{\bar{e}}^*(\rho_v^\sharp) \|_{G_e} \\ 
    &=  \| ((\rho_u \circ i_e)  \restrict_{\rho(s_e)(Y_e)})^\sharp + ((\rho_u \circ i_e)  \restrict_{X \setminus \rho(s_e)(Y_e)})^\sharp \\ & \quad \quad \quad  \quad - ((\rho_v \circ i_{\bar{e}})  \restrict_{ Y_e})^\sharp - ((\rho_v \circ i_{\bar{e}})  \restrict_{X \setminus Y_e})^\sharp \|_{G_e} \\
    &\le \| i_e^*(\rho_u\restrict_{X - \rho(s_e)(Y_e)})^\sharp ) \|_{G_e} + \|i^*_{\bar{e}}(\rho_v\restrict_{X - Y_e})^\sharp )\|_{G_e}  \\
    &= |X-s_e(Y_e)| + |X-Y_e|  \le 2\delta|G_e|^2|X|.
\end{split}
\end{align*}
Averaging the above estimate over all oriented edges $ e\in \vec{E}(\Gamma)$ gives 
$$\|\DG (\rho^\sharp) \|_E \le c\delta|X|= c\delta\|\rho^\sharp\|_V$$ 
with respect to the  constant  $c=2\max_{e\in\vE(\Gamma)} |G_e|^2$. This constant  depends only on the graph of groups $\calG$.
\end{proof}




\section{Linear algebra and cones}
\label{sec: linear cones}

This section is, formally speaking, independent of the rest of the paper. Its goal is to show that \enquote{$\bbZ$-linear maps are stable}, in the sense that an approximate solution to a system of linear equations and inequalities must be close to an exact $\bbZ$-solution (see Lemma \ref{lem:linear correction to a smaller vector} below for a precise statement).

\vspace{10pt}

Let $\Lambda_1$ and $\Lambda_2$ be a pair of finitely generated free $\bbZ$-modules. Let $\DD : \Lambda_1 \to \Lambda_2$ be a $\bbZ$-linear map.

Let $V_i = \Lambda_i \otimes \bbR$ be the $\bbR$-vector spaces obtained by extending scalars from $\Lambda_i$ and  $\|\cdot \|_i$ be norms on $V_i$ for $i = 1,2$. By abuse of notation, we continue using $\DD : V_1 \to V_2$ to denote  the $\bbR$-linear extension of $\DD : \Lambda_1 \to \Lambda_2$. 
We will make essential use of the fact that $\DD : V_1 \to V_2$ is defined over $\bbQ$.
Denote $K = \ker \DD$ so that $K $ is a  $\bbQ$-subspace of the $\bbR$-vector space $V_1$.

Assume that $C \subset V_1$ is a  closed positive cone defined by finitely many   inequalities over $\bbQ$ and satisfying  $\Span(C) = V_1$. 
Denote $\Lambda_1^+ = C \cap \Lambda_1$ so that the  subset $\Lambda_1^+$ is closed under addition.

\begin{lemma}\label{lem: approximating kernel in cones}
For all $v\in C$ there exists $v''\in C\cap K$ such that $\|v-v''\|_1 \prec \|\DD v\|_2$.
\end{lemma}

We point out that the intersection $C \cap K$ is non-empty for it contains the zero vector $0 \in V_1$.
 Lemma \ref{lem: approximating kernel in cones} does not require the assumption   that the subspace $K$,  the linear map $\DD $ and the positive cone $C$ are all defined over $\bbQ$. We do need however the assumption that $C$ is defined by finitely many inequalities.

\begin{proof}[Proof of Lemma \ref{lem: approximating kernel in cones}]
 We argue by induction on the $\bbR$-dimension of $V_1$. The base case where $\dim_\bbR V_1 = 0$ is trivial.

Consider the $\bbR$-subspace $\DD(V_1)$  of the $\bbR$-vector space equipped  with two different norms, namely the restriction of norm $\|\cdot \|_2$ coming from $V_2$, and the quotient norm $\|\cdot\|'_1$ defined by
$$\|\DD v\|_1' := \inf_{w\in K} \|v-w\|_1 \quad \forall v \in V_1.$$ 
Since any two norms on a finite dimensional $\bbR$-vector space are bi-Lipschitz equivalent, there is some constant $c>0$ such that $\|\DD v\|_1' \le c \|\DD v\|_2$ for all $v \in V_1$. 

Fix some vector $v \in C$. The infimum appearing in the definition of the quotient norm $\|\DD v\|_1'$ is attained at some vector $w\in K$, hence 
$$\|\DD v\|'_1 = \|v-w\|_1 \le c \|\DD v\|_2.$$

If $w\in C$, then we are done by choosing $v'=w \in C \cap K$. 
Otherwise, let $u$ be the closest point to $w$  along the closed segment $[v,w] \subseteq V_1$ and belonging to the closed cone $C$. Then clearly 
$$\|v-u\|_1\le \|v-w\|_1 = \|\DD v\|'_1 \le c\|\DD v\|_2 \quad \text{and} \quad  \|\DD u\|_2 \le \|\DD v\|_2.$$ 
Since the point $u$ lies on the boundary of the positive cone $C$, it belongs to some proper face $D\subset C$ spanning a $\bbQ$-subspace $U_1 = \Span_\bbR(D) \subset V_1$ of strictly lower dimension. 
By the induction hypothesis there exists some constant $c_D$, such that for $u\in D$ there exists a point $v' \in D \cap K$ with $\|u-v'\|_1 \le c_D \|\DD u\|_2$. 
Hence, $$\|v-v'\|_1 \le  \|v-u\|_1 + \|u - v'\|_1 \le  (c+c_D)\|\DD v\|_2 \le c_1 \|\DD v\|_2$$
where $c_1 = c+\max_{D\subset C}c_D$ and the maximum is taken over the finite set of proper faces of the positive cone $C$.
\end{proof}

Recall that $K$ denotes the kernel of the linear map $\DD$ regarded as a $\bbQ$-subspace of the $\bbR$-vector space $V_1$. 

\begin{lemma}
\label{lem:affine bound}
There are constants $c_1,A > 0$ such that if $v \in C$ then there is a vector $\lambda \in \Lambda^+ \cap K $  satisfying $\|v-\lambda\|_1 \le c_1\|\DD v\|_2 + A$.
\end{lemma}
\begin{proof}
 Let $v \in C$ be any vector.  By Lemma \ref{lem: approximating kernel in cones} there exists a vector $v''\in C \cap K$ such that $\|v-v''\|_1 \le c_1  \|\DD  v\|_2$ for some constant $c_1 > 0$ independent of $v$.
 
 Note that $C \cap K$ is a positive cone defined over $\bbQ$. Let $U = \Span_\bbR(C \cap K) \subseteq V_1$ so that $C\cap U$ is a closed cone with a non-empty interior in the vector subspace $U$.  Denote $B_A = \{w \in V_1 \: : \: \|w\|_1 \le A\}$. Therefore   $C \cap U \cap B_A$ contains in its interior a ball in $U$ of an arbitrary large radius, provided the radius  $A>0$ is sufficiently large. Since $\Lambda \cap U$ is a lattice in the $\bbR$-vector space $U$,  the set $C\cap U \cap B_A$ surjects onto $U/(U \cap \Lambda)$ for all  $A > 0$ sufficiently large. Fix any  sufficiently large such  $A>0$.

Since $v''\in C \cap K$ the translated set $v''+C\cap U \cap B_A \subset C \cap K$ also surjects onto $U/(U \cap \Lambda)$. In particular this set admits a point    $\lambda\in \Lambda \cap C \cap K = \Lambda^+ \cap K$.
We conclude that $\|v-\lambda\|_1 \le \|v-v''\|_1 + \|v'' - \lambda\|_1 \le c_1 \|\DD  v\|_2 + A$ as required.
\end{proof}

\begin{lemma}
\label{lem:linear correction to a smaller vector}
For any vector $\lambda \in \Lambda^+$ there is another vector $\lambda' \in \Lambda^+ \cap K$   satisfying $\|\lambda-\lambda'\|_1 \prec \|\DD  \lambda\|_2 $ and $\|\lambda'\|_1 \le \|\lambda\|_1$.
\end{lemma}
\begin{proof}
Let the vector  $\lambda  \in \Lambda^+$ be fixed. If $\lambda  \in K = \ker \DD $ then there is nothing to prove, for we may simply take $\lambda ' = \lambda \in \Lambda^+ \cap K$. Assume therefore that $\lambda \notin K$.

Since the linear map $\DD $ is defined over $\bbQ$ there is a constant $M > 0$ such that $\|\DD \lambda \|_2 \ge M$ for every vector $\lambda  \in \Lambda \setminus K $.    
Denote $$\theta = \frac{c_1\|\DD \lambda \|_2+A}{\|\lambda\|_1}$$ 
where the constants $c_1$ and $A$ are as in Lemma \ref{lem:affine bound}.

If $\theta \ge 1$ then we may take $\lambda'=0$. This vector  $\lambda'$ satisfies  $0 = \|\lambda'\|_1  \le \|\lambda\|_1$ and
$$   \|\lambda - \lambda'\|_1 \le \theta  \|\lambda\|_2 = c_1\|\DD\lambda\|_2+A \le (c_1 +\frac{A}{M})\|\DD \lambda\|_2$$
as desired (the constants $c_1, A$ and $M$ are all independent of $\lambda$).

Finally assume  that $0<\theta<1$. Apply Lemma \ref{lem:affine bound} to the vector $v=(1-\theta) \lambda $. This gives a new vector $\lambda ' \in \Lambda^+ \cap K $ with 
$$\|v -\lambda '\|_1  \le c_1   \|\DD v\|_2 + A \le c_1\|\DD \lambda \|_2 + A.$$
Therefore
$$ \| \lambda ' \|_1 \le \|  v \|_1 + \| v  - \lambda '\|_1 \le  (1-\theta)\|\lambda \|_1 + c_1 \|\DD \lambda \|_2 + A =  \|\lambda \|_1. $$
This verifies the second condition. As for the first condition,  we have
$$ \|\lambda -\lambda '\|_1 \le \|\lambda -v \|_1 + \|v  - \lambda '\|_1 \le \theta\|\lambda \|_1 + c_1\|\DD \lambda \|_2+A \le 2c_1 \|\DD \lambda \|_2+2A \le 2(c_1 + A/M) \|\DD \lambda\|_2.$$
This concludes the proof, noting as above that the constants $c_1, A$ and $M$ are all independent of the vector $\lambda$.
\end{proof}

\section{From linear algebra back to actions}
\label{sec: from linear algebra to actions}

We show that any $\delta$-almost $\piG$-action $\rho$ whose isomorphism type $\rho^\sharp$ is compatible with some $\piG$-action, can be corrected to such an action. More precisely, we establish the following result, using without further mention all the notations introduced in \S\ref{sec: relative stability}, \S\ref{sec:fundamental group} and \S\ref{sec: from actions to linear algebra}.

\begin{proposition}\label{prop: ker implies action} 
Let   $\rho : \barpiG \to \Sym(X) $ be a $\delta$-almost $\piG$-action with $\lambda = \rho^\sharp$. Let $\lambda' \in\Lambda_V^+$ be any vector with $\|\lambda'\|_V = \|\lambda\|_V$. If  
\begin{enumerate}[label=(\alph*)]
    \item $\lambda' \in \ker \DG$ and
    \item $\|\lambda - \lambda' \|_V \le \delta \|\lambda\|_V$ 
\end{enumerate}
then there is a group action $\rho' : \piG \to \Sym(X) $ satisfying
\begin{enumerate}[label=(\roman*)]
\item $\lambda' = (\rho')^\sharp  $ and 
\item $d_{X,S_\calG}(\rho,\rho') \prec_\calG \delta$. 
\end{enumerate}
\end{proposition}

We precede the proof of Proposition \ref{prop: ker implies action} with an analogous statement in the simpler context of a single group homomorphism.

\begin{lemma}
\label{lemma:fixing G wrt H}
Let $i:H\to G$ be a homomorphism of finite groups.
Let $\phi:H \to \Sym(X)$ and $\rho:G\to \Sym(X)$ be a pair of group actions. Denote $\lambda = \rho^\sharp$.
If  $\lambda'\in \Lambda_G^+$ and $\delta > 0$ are such that 
\begin{enumerate}[label=(\alph*)]
    \item $d_{X,H}(\rho\circ i, \phi)\le \delta$,
    \item $i^*(\lambda')=\phi^\sharp$, and
    \item $\|\lambda - \lambda'\|_G \le \delta \|\lambda\|_G$
\end{enumerate} 
then there exists a group action $\rho':G\to \Sym(X)$ satisfying
\begin{enumerate}[label=(\roman*)]
    \item $\rho'\circ i = \phi$,
    \item $\lambda' = (\rho')^\sharp  $, and
    \item $d_{X,G}(\rho,\rho') \prec_{G,H} \delta$.
\end{enumerate}
\end{lemma}
Note that the \enquote{small} action $\phi$ of the group $H$ is not being changed, rather the \enquote{large} action $\rho$ is being replaced with a new action $\rho'$ compatible with $\phi$.

\begin{proof}[Proof of Lemma \ref{lemma:fixing G wrt H}]
We combine  Assumption (a) and Observation \ref{obs: existence of invariant subsets} applied with respect to the finite group $H$ in order to find a $\phi(H)$-invariant subset $X_0 \subset X$ satisfying $\phi \restrict _{X_0} = (\rho \circ i) \restrict _{X_0}$ and $|X_0| \ge (1-\delta|H|)  |X|$. By applying Observation \ref{obs: existence of invariant subsets} a second time with respect to the finite group $G$, we find a $\rho(G)$-invariant subset $X_1 \subset X_0$ satisfying $|X_1| \ge (1-\delta|H||G|) |X|$. 

Consider the vector $\lambda _1 = (\rho \restrict _{X_1})^\sharp \in \Lambda_G^+$. Let  $\mu_1  \in \Lambda_G^+$ be the component-wise minimum of the two vectors $\lambda' $ and $\lambda _1$, i.e $\mu_1$ is the vector given by
$$ (\mu_1)_\chi = \min \{(\lambda') _\chi, (\lambda _1)_\chi \} \quad \forall \chi \in \Trans(G).$$
The previous paragraph implies that $\|\lambda - \lambda _1\|_G \le \delta |H||G| |X|$. Therefore Assumption (c)   gives
\begin{align*}   
 \max\{ \|\lambda' -\mu_1\|_G, \|\lambda _1-\mu_1\|_G\}  &\le \|\lambda' -\lambda _1\|_G  
 \le \|\lambda' -\lambda \|_G + \|\lambda - \lambda _1\|_G \\ 
 &\le \delta |X| + \delta|H||G| |X| = \delta (1+|H||G|)|X|.
\end{align*}

Let $Y_1 \subset X_1$ be any $\rho(G)$-invariant subset satisfying $\mu_1 = (\rho\restrict _{Y_1})^\sharp $. 
Write $\mu_2=\lambda'  - \mu_1 \in \Lambda_G^+$ and $Y_2 = X \setminus Y_1$ so that  $\lambda'  = \mu_1 + \mu_2$ and $X = Y_1 \coprod Y_2$. It will not be the case in general that   $(\rho\restrict Y_2)^\sharp$ coincides with $\mu_2$. However   $|Y_2| = \|\mu_2\|_G$ and  in particular the size of the subset $Y_2$ is bounded by
$$|Y_2| = \|\mu_2\|_G \le \delta (1+|H||G|) |X|.$$

We   define a new action $\rho' : G \to \Sym(X)$ as follows. 
To begin with, the restriction of $\rho'$ to the $\rho(G)$-invariant subset $Y_1$ coincides with $\rho$, namely   $\rho' \restrict _{Y_1} = \rho \restrict _{Y_1}$. As $i^*(\lambda')=\phi^\sharp$ by Assumption (b) and as $i^*(\mu_1) =(\phi\restrict_{Y_1})^\sharp$ by the choice of the subset $Y_1$ we have $i^*(\mu_2) = (\phi\restrict_{Y_2})^\sharp$. It remains to define the  action $\rho'$ on the $\rho(G)$-invariant complement $Y_2 = X \setminus Y_1$.
Taking into account  Observation \ref{obs: extending actions from subgroups} we  let  $\rho' \restrict _{Y_2}$ be an arbitrary action satisfying  $(\rho' \circ \iota) \restrict _{Y_2} = \phi \restrict _{Y_2}$ and $(\rho'\restrict_{Y_2})^\sharp = \mu_2$. This completes the definition of the new action $\rho'$. 

Statements (i) and (ii) of the Lemma hold true since  $\rho' \circ i = \phi$ and $$\rho'^\sharp = (\rho'\restrict_{Y_1})^\sharp + (\rho'\restrict_{Y_2})^\sharp = \mu_1 + \mu_2 = \lambda'. $$ 
To verify Statement (iii)   we compute
\begin{align*}
\begin{split}
d_X(\rho(g),\rho'(g)) &= \frac{|Y_1|}{|X|}d_{Y_1} (\rho (g) \restrict _{Y_1},\rho' (g)\restrict_{ Y_1})  + \frac{|Y_2|}{|X|} d_{Y_2} (\rho(g)\restrict _{Y_2},\rho'(g)\restrict_{ Y_2} ) \\& \le \frac{|Y_2|}{|X|} \le (1+|H||G|)\delta \le 2 |H||G| \delta
\end{split}
\end{align*}
for all elements $g\in G$. Therefore  $d_{X,G}(\rho,\rho') \le 2 |H||G|^2 \delta$ as required.
\end{proof}

We are now in a position to prove the main result of \S\ref{sec: from linear algebra to actions}.

\begin{proof}[Proof of Proposition \ref{prop: ker implies action}]
We define the new action $\rho' : \piG \to \Sym(X) $ of the fundamental group $\piG$ by specifying it on the finite generating set $S_\calG$. This is done in two steps: first we define $\rho'$ on the vertex groups $G_v$ and then on the generators of the form $s_e$.

\subsubsection*{Step 1. Defining $\rho'$ on $G_v$ for all $v\in V(\Gamma)$.}
Fix an arbitrary base  vertex $v_0$ in $ V(\Gamma)$. We define the vertex group actions $\rho'|_{G_v}$ by induction on the distance in the spanning tree $T$ of the vertex $v$ from the base vertex $v_0$  such that:
\begin{enumerate}
    \item $(\rho'|_{G_v})^\sharp = \lambda'_v$,
    \item $d_{X,G_v}(\rho|_{G_v}, \rho'|_{G_v}) \prec \delta$ and 
    \item $\rho|_{G_{t(e)}}\circ i_e = \rho|_{G_{o(e)}}\circ i_{\bar{e}}$ for the unique edge $e\in E(T)$ such that $t(e)=v$ and the unique path in the tree $T$ from $v_0$ to $v$  passes through $o(e)$.
\end{enumerate}


\emph{Base of the induction.} We apply Lemma \ref{lemma:fixing G wrt H} with the following data: the group $G$ is the vertex group $G_{v_0}$, the group $H$ and the homomorphism $ i : H \to G$ are trivial and the vector $\lambda' \in \Lambda_G^+$ is the coordinate $\lambda'_{v_0} \in \Lambda_{G_{v_0}}^+$.
This results in a new  action $\rho'|_{G_{v_0}}$ of the base vertex group $G_{v_0}$ satisfying 
 $d_{X,G_{v_0}}(\rho|_{G_{v_0}},\rho'|_{G_{v_0}}) \prec \delta$
and $(\rho'|_{G_{v_0}})^\sharp = \lambda'_{v_0}$.

\vspace{5pt}

\emph{Induction step.} Let $v\in V(\Gamma)$ be a vertex of distance $n \in \bbN$ from the base vertex $v_0$ in the tree $T$ and $e\in E(T)$ be the unique edge such that $t(e)=v$ and $o(e)$ is distance $n-1$ from the base vertex $v_0$. Denote $u = o(e)$. 
By the induction hypothesis the vertex group action $\rho'|_{G_u}$ has been defined and satisfies $(\rho'|_{G_u})^\sharp = \lambda'_u$. 

We  apply Lemma \ref{lemma:fixing G wrt H} with the following data: the group $G$ is the vertex group $G_{v}$, the group $H$ is the edge group $G_e$, the homomorphism $ i : H \to G$ is the map $i_e$, the action $\phi$ of the group $H$ is $(\rho'|_{G_u})\circ i_{\bar{e}}$, the action  $\rho$ of the group $G$  is $\rho|_{G_v}$ and lastly the vector $\lambda' \in \Lambda_G^+$ is the coordinate $\lambda'_v \in \Lambda_{G_v}^+$. 

We proceed to verify the assumptions of Lemma \ref{lemma:fixing G wrt H}. 
The induction hypothesis combined with the assumption that  $ \lambda' \in \ker \DG$     imply 
$$i^*(\lambda'_v) = i_e^*(\lambda'_v) = i^*_{\bar{e}}(\lambda'_u) = i^*_{\bar{e}}((\rho'|_{G_u})^\sharp) = (\rho' \circ i_{\bar{e}})^\sharp = \phi ^\sharp. $$
By the triangle inequality the two actions $\rho \circ i$ and $\phi$ of the edge group $G_e$ satisfy
\begin{align*}
 \begin{split}
     d_{X,G_e}( \rho \circ i, \phi) & = d_{X,G_e}( \rho \circ i_e, \rho' \circ i_{\bar{e}}) \\
    &\le d_{X,G_e}( \rho \circ i_e, \rho(s_e) \cdot (\rho \circ i_e))  
    \\&\quad+d_{X,G_e}( \rho(s_e) \cdot (\rho   \circ i_e), 
    \rho(s_e) \cdot (\rho   \circ i_e) \cdot \rho(s_e)\ii) 
    \\&\quad+ d_{X,G_e}(\rho(s_e) \cdot (\rho   \circ i_e) \cdot \rho(s_e)\ii,\rho \circ i_{\bar{e}}) 
     \\&\quad+ d_{X,G_e}( \rho \circ i_{\bar{e}}, \rho' \circ i_{\bar{e}}).
\end{split}
 \end{align*}
The normalized Hamming metric $d_X$ is bi-invariant so  that the first and second summands are both less than $d_{X,G_e}(\rho(s_e), \id) < \delta$.
The third summand  is also less than $\delta$ as $\rho$ is a $\delta$-almost $\piG$-action and taking into account  the corresponding relation in $R_\calG$. Lastly, the fourth summand satisfies 
$d_{X,G_e}(\rho \circ i_{\bar{e}}, \rho' \circ i_{\bar{e}}) \prec \delta$
by the induction hypothesis. We conclude that
\begin{align*}
d_{X,G_e}( \rho \circ i, \phi) \prec \delta.
\end{align*}

Having verified all of the assumptions for Lemma \ref{lemma:fixing G wrt H}, we get a new action $\rho'|_{G_v}$ of the vertex group $G_v$ such that $\rho'\circ i_e = \rho' \circ i_{\bar{e}}$ on the edge group $G_e$, 
$d_{X,G_v}(\rho|_{G_v},\rho'|_{G_v})\prec \delta$ 
and $(\rho'|_{G_v})^\sharp = \lambda'_v$. This completes the step of the induction.

\vspace{5pt}

\emph{Proceed with the induction until the new action $\rho'$ is defined on all vertex groups.}

\subsubsection*{Step 2. Defining $\rho'$ on the generators $s_e$ for all $e \in \vE(\Gamma)$}
Let $e \in \vE(\Gamma)$ be a directed edge with $o(e) = u$ and $t(e) = v$. 

\vspace{5pt}
\emph{Assume that $e \in E(T)$.} Define   $\rho'(s_e) = \id$. Recall that the action $\rho'$ of the edge group $G_e$ satisfies  $\rho'\circ i_e = \rho' \circ i_{\bar{e}}$ by Step 1. Therefore
\begin{equation*}
\rho'( i_e(g_e)s_e) (x) = \rho'( s_e i_{\bar{e}}(g_e) )(x)
\end{equation*}
for all points $x\in X$ and all elements $g_e\in G_e$. Moreover, since $\rho$ is a $\delta$-almost $\piG$-action we have  $d_X(\rho(s_e),\rho'(s_e)) \le \delta$.

\vspace{5pt}
\emph{Assume that $e \in E(\Gamma)\setminus E(T)$.}
According to  Observation \ref{obs: existence of invariant subsets} there exists a $\rho'|_{i_e(G_e)}$-invariant subset $X_e \subseteq X$, such that 
$|X-X_e| \prec \delta |X|$ 
and the following conditions are satisfied for all points $x\in X_e$ and all elements $g_e \in G_e$ 
\begin{align*}
\begin{split}
    \rho( i_e(g_e)s_e) (x) &= \rho( s_e i_{\bar{e}}(g_e)  )(x),\\
    \rho(i_{\bar{e}}(g_e))(x)&=\rho'(i_{\bar{e}}(g_e))(x),\\
    \rho(i_{e}(g_e))(\rho(s_e) x) &= \rho'(i_{e}(g_e))(\rho(s_e) x).
    \end{split}
\end{align*} 
Define the restriction $\rho'(s_e) \restrict|_{X_e}$ of the new action to be the same as $\rho(s_e) \restrict _{X_e}$. The above conditions imply that the permutation   $\rho'(s_e)$ satisfies
\begin{equation*}
\rho'(i_e(g_e)s_e ) (x) = \rho'( s_e i_{\bar{e}}(g_e) )(x)
\end{equation*}
for all points $x\in X_e$ and all edge group elements  $g_e \in G_e$.

It remains to define the permutation $\rho'(s_e)$ on the complement $X-X_e$  and verify the above relation for all points  $x\in X-X_e$. The two actions $\rho' \circ i_e$ and $\rho' \circ i_{\bar{e}}$ of the edge group $G$ are conjugate as   $\lambda'\in \ker \DG$ and $\rho'^\sharp = \lambda'$.
 Since the permutation $\rho(s_e)$ conjugates $(\rho' \circ i_e) \restrict _{X_e}$ to $(\rho' \circ i_{\bar{e}}) \restrict _{\rho(s_e) X_e}$, we know that their complements must be conjugate as well. Define the restriction $\rho'(s_e) \restrict_{X \setminus X_e}$ to be an arbitrary bijection from $X \setminus X_e$ to $X \setminus \rho(s_e)X_e$ implementing this isomorphism of actions. Note that 
$d_X(\rho(s_e), \rho'(s_e)) \prec \delta$.
This concludes the definition of the permutation $\rho'(s_e)$ for this particular oriented edge $e$.

\subsubsection*{A bound on $d_{X,S_\calG}(\rho,\rho')$.} The $\piG$-action $\rho'$ has been constructed in   Steps 1 and 2. It was specified in terms of the finite generating set $S_\calG$ while making sure that the defining relations $R_\calG$ of the fundamental group $\piG$ hold true. It follows from the construction that   $\rho'^\sharp = \lambda'$. 
To conclude the proof it remains to bound the normalized Hamming distance $d_{X,S_\calG}(\rho,\rho')$.  Namely
\begin{align*}
\begin{split}
d_{X,S_\calG}(\rho,\rho') &= \sum_{\sigma \in S_\calG} d_X(\rho(\sigma), \rho'(\sigma))  \\
&=   \sum_{v \in V(G)} \sum_{g \in G_v}   d_X(\rho(g), \rho'(g))   \\
&\quad\quad+   \sum_{e \in E(T)}  d_X(\rho(s_e), \rho'(s_e)) \\
&\quad\quad+     \sum_{e \in E(\Gamma) \setminus E(T)}  d_X(\rho(s_e), \rho'(s_e))  \\
&\prec \delta
\end{split}
\end{align*}
as required.
\end{proof}


\section{Proof of the main theorem}

We are ready to show that the epimorphism $\barpiG \to \piG$ is P-stable.

\begin{proof}[Proof of Theorem \ref{thm:quotient map is stable}]
Let $X$ be a finite set admitting a $\delta $-almost $\piG$-action $\rho:\overline{\pi}_1(\calG,T) \to \Sym(X)$. Denote $\lambda = \rho^\sharp$.  We know by   Proposition \ref{prop:a delta almost action gives almost kernel} that
\begin{equation}
\|\DG (\lambda)\|_E \prec_{\calG} \delta \| \lambda \|_V
\end{equation}
Lemma \ref{lem:linear correction to a smaller vector} allows us to find a vector $\lambda'' \in \Lambda_V^+ \cap \ker \DG$ such that 
\begin{equation*}
\|\lambda'' - \lambda\|_V \prec \delta \|\lambda\|_V \quad \text{and} \quad \|\lambda''\|_V \le \|\lambda\|_V.
\end{equation*}

We will make an auxiliary use of the action of the fundamental group $\pi_1(\calG,T)$ on a singleton. Denote this action by $s$. By Proposition \ref{prop:honest action is in kernel} we know that $s^\sharp \in \ker \DG$. Moreover $\|s^\sharp\|_V = 1$. 
Let 
$$ \lambda'  = \lambda'' + (\|\lambda\|_V - \|\lambda''\|_V) s^\sharp.$$
It is clear that  $\lambda'  \in \ker \DG$, $\|\lambda' \|_V = \|\lambda\|_V = \|\rho^\sharp\|_V = |X|$ and
$$ \|\lambda'  - \lambda\|_V \le \|\lambda - \lambda''\|_V + \|\lambda' - \lambda'' \|_V  \prec \delta |X|.$$ 

To conclude the proof we apply Proposition \ref{prop: ker implies action} and obtain the desired action $\rho' : \piG  \to \Sym(X)$ satisfying $(\rho')^\sharp = \lambda' $ and
 $d_{X,S_\calG}(\rho, \rho') \prec \delta$. 
\end{proof}

\begin{remark}
It follows from the proof that one can take $\delta \prec \epsilon$ for the P-stability of $\piG$.
\end{remark}

The derivation of Theorem \ref{thm: main thm} from the above Theorem \ref{thm:quotient map is stable} is immediate and has been discussed in \S\ref{subsec: outline of proof}.

\section{Graph automorphisms of finite order}
\label{sec:graph automorphisms}

Fix some $d \in \bbN$ and let $F_d=F(s_1,\ldots,s_d)$ be the free group of rank $d$.

A finite Schreier graph  $A$ of the group $F_d$ is a finite directed graph edge-labelled by the generators $s_1,\ldots,s_d$ such that every vertex has exactly one incoming and one outgoing edge of each label. We indicate the labelling of the directed edges $E(A)$ using a function  $c=c_A:E(A) \to \{s_1,\ldots,s_d\}$.

A \emph{weak $\delta$-almost-automorphism} $\alpha$ of the Schreier graph $A$ is a pair of bijections $\alpha: V(A)\to V(A)$ and   $\alpha:E(A)\to E(A)$ (we use the same letter for both by abuse of notation) such that for all directed edges $e\in \vE(A)$ except for a subset of size $\delta|E(A)|$ we have
 $$c(\alpha(e)) = c(e), \; o(\alpha(e))=\alpha(o(e)) \quad \text{and} \quad t(\alpha(e)) = \alpha(t(e)).$$
 
 A \emph{$\delta$-almost-automorphism} $\alpha$ is a weak $\delta$-almost-automorphism that moreover satisfies the first two conditions, namely $c(\alpha(e)) = c(e)$ and $o(\alpha(e))=\alpha(o(e))$, for all directed edges $e\in E(A)$.

\begin{observation}
Let $\alpha$ be a weak $\delta$-almost-automorphism of the finite Schreier graph $A$. Up to changing $\alpha$ on at most $O(\delta|E(A)|)$ edges we can make $\alpha$ into a $\delta$-almost-automorphism.
\end{observation}
Fix some integer $n \in \bbN$.
\begin{definition}
A (weak) $\delta$-almost-automorphism $\alpha$ has \emph{$\delta$-almost order $n$} if the condition $\alpha^n(v)=v$ holds true for all $v\in V(A)$ except for a subset of size $\delta |V(A)|$.
\end{definition}

Given an action $\rho:F_d *\gen{a}\to \Sym(X)$ on some finite set $X$ we denote by $A_\rho$ the Schreier graph of the restricted action $\rho:F_d \to\Sym(X)$. Let $\alpha_\rho$ denote the bijection on the vertices of the Schreier graph $A_\rho$ defined by $\alpha_\rho=\rho(a)$. Moreover, by abuse of notation, let $\alpha_\rho$ denote the bijection of the directed edges of $A_\rho$ defined for every $e \in E(A)$ by   $\alpha_\rho(e) = e'$ where $e'$ is the unique edge satisfying $c(e)=c(e')$ and $o(e') = o(\alpha_\rho(e))$.


\begin{observation}
    If $\rho:F_d *\gen{a}\to \Sym(X)$ is a $\delta$-almost $(F_d \times \bbZ)$-action (resp. $\delta$-almost $(F_d \times (\bbZ/n\bbZ))$-action) on some finite set $X$ then $\alpha_\rho$ is a $\delta$-almost-automorphism of the Schreier graph $A_\rho$ (resp. of $\delta$-almost-order $n$). 
    
    Vice versa, if $A$ is a finite Schreier graph of the group $F_d$ and $\alpha$ is $\delta$-almost-automorphism  (resp. of $\delta$-almost-order $n$) of the free group $F_d$ then there exists a $\delta$-almost $(F_d \times \bbZ)$-action (resp. $\delta$-almost $(F_d \times (\bbZ/n\bbZ))$-action) $\rho:F_d *\gen{a}\to \Sym(X)$ such that $A=A_\rho$ and $\alpha = \alpha_\rho$.
\end{observation}

 Theorem \ref{thm: main thm} applied to the virtually free group $F_d\times (\bbZ/n\bbZ)$ and combined with the above observations immediately gives the following corollary.

\begin{corollary}
\label{cor:almost period automorphisms}
    Let $A$ be a finite Schreier graph of the free group $F_d$ and let $\alpha$ be a weak $\delta$-almost automorphism of $\delta$-almost order $n$. Then there exists a Schreier graph $A'$ of the group $F_d$ with $V(A) = V(A')$, and an automorphism $\alpha'$ of $A'$ of order $n$ such that the graphs $A$ and $A'$ differ on at most $O(\delta|E|)$ edges, and the automorphisms $\alpha$ and $\alpha'$ differ on at most $O(\delta|V|)$ vertices.\qed
\end{corollary}
    
Note that Corollary \ref{cor:almost period automorphisms} is false without requiring that $\alpha$ has $\delta$-almost order $n$ since $F_d\times \bbZ$ is non P-stable by \cite{ioana2020stability}.

We end this paper with the following related question. 

\begin{question}
Is the conclusion of Corollary \ref{cor:almost period automorphisms} true in the setting of general $d$-regular graphs and graph automorphisms (rather than Schreier graphs of $F_d$)?
\end{question}

\bibliographystyle{plain}
\bibliography{biblio}
\end{document}